\pdfoutput=1 
\newif\ifpreprint
\preprinttrue
\ifdefined\ispreprint
  \preprinttrue
 \fi
\ifdefined\isjournal
  \preprintfalse
\fi

\ifpreprint 

\documentclass{article}

\usepackage{mathtools}
\usepackage{mathrsfs}
\usepackage{todonotes}
\usepackage{hyperref}
\usepackage{color,graphicx}
\usepackage{amsfonts}
\usepackage{amssymb}
\usepackage{amsthm}
\usepackage{amsmath}
\usepackage[normalem]{ulem}
\usepackage{comment}
\usepackage{mathrsfs} 

\newtheorem{theorem}{Theorem}[section]
\newtheorem{lemma}[theorem]{Lemma}

\newtheorem{remark}[theorem]{Remark}
\newtheorem{proposition}[theorem]{Proposition}
\newtheorem{definition}[theorem]{Definition}
\newtheorem{corollary}[theorem]{Corollary}

  \numberwithin{equation}{section}

\newenvironment{@abssec}[1]{%
     \if@twocolumn
       \section*{#1}%
     \else
       \vspace{.05in}\footnotesize
       \parindent .2in
         {\upshape\bfseries #1. }\ignorespaces 
     \fi}
     {\if@twocolumn\else\par\vspace{.1in}\fi}
\newcommand{\email}[1]{\protect\href{mailto:#1}{#1}}
\newenvironment{keywords}{\begin{@abssec}{Key words}}{\end{@abssec}}
\newenvironment{MSCcodes}{\begin{@abssec}{MSC codes}}{\end{@abssec}}
\else 

\fi

\newcommand{\ri}{\mathrm{i}}
\newcommand{\re}{\mathrm{e}}
\newcommand{\rd}{\, \mathrm{d}}
\newcommand{\C}{\mathbb{C}}

\newcommand{\N}{\mathbb{N}}

\newcommand{\T}{\mathbb{T}}
\newcommand{\R}{\mathbb{R}}

\newcommand{\Z}{\mathbb{Z}}

\newcommand{\calO}{\mathcal{O}}

\usepackage[backend=biber,style=numeric,maxbibnames=99,giveninits=true,isbn=false,url=false]{biblatex}
\DeclareFieldFormat*{title}{\mkbibemph{#1}}
\DeclareFieldFormat*{citetitle}{\mkbibemph{#1}}
\DeclareFieldFormat{journaltitle}{#1}

\AddToHook{package/biblatex/after}{%
	\renewbibmacro{in:}{}%
	\AtEveryBibitem{
		\clearfield{day}
		\clearfield{month}
		\clearfield{number}
		\clearlist{location}}
}

\newbibmacro*{pubinstorg+location+date}[1]{%
  \printlist{#1}%
  \newunit
  \printlist{location}%
  \newunit
  \usebibmacro{date}%
  \newunit}

\renewbibmacro*{publisher+location+date}{\usebibmacro{pubinstorg+location+date}{publisher}}
\renewbibmacro*{institution+location+date}{\usebibmacro{pubinstorg+location+date}{institution}}
\renewbibmacro*{organization+location+date}{\usebibmacro{pubinstorg+location+date}{organization}}

\usepackage{xcolor}
\definecolor{darkblue}{RGB}{0,60,180} 
\definecolor{darkgreen}{RGB}{0,130,70}
\definecolor{darkorange}{RGB}{180,60,0}

\allowdisplaybreaks

\ifpreprint
\title{M{\"o}bius-Transformed Trapezoidal Rule for Polynomial Weights}
\else
\fi

\author{Nuutti Hyv\"onen\thanks{Department of Mathematics and Systems Analysis, School of Science, Aalto University, Espoo, FI-00076 Aalto, Finland
  (\email{nuutti.hyvonen@aalto.fi}).}
  \and Yuya Suzuki\thanks{Corresponding author. Department of Mathematics and Systems Analysis, School of Science, Aalto University, Espoo, FI-00076 Aalto, Finland
  (\email{yuya.suzuki@aalto.fi}).}}

\begin{document}

\maketitle
\begin{abstract}
    This work studies numerical integration by the Möbius-transformed trapezoidal rule, which combines the classical trapezoidal rule with a change of variables induced by a M\"obius transformation that maps the unit circle onto the real line. It is shown that this method achieves the optimal convergence rate for a polynomially weighted integral over the real line if the integrand lives in a related polynomially weighted Sobolev space with positive integer smoothness index. This result can also be generalized in a slightly weaker form for fractional smoothness indices via complex interpolation of function spaces. The algorithm only requires pointwise evaluations of the weight and the target integrand at prescribed nodes that do not depend on the integrand and weight in question. The established  theoretical convergence rates are verified by numerical experiments.
\end{abstract}



\section{Introduction}

In this paper, we consider numerical integration over the real line with smooth positive weight functions that exhibit polynomial behavior at infinity, with
\begin{equation}
\label{eq:weight0}
\omega_{\upsilon}(x)= \big(1+ x^2 \big)^{-\upsilon/2}, \quad \upsilon \in \R,
\end{equation}
as the generic representative for such a class of weights. 
The reason for choosing the negative sign for $\upsilon$ in the exponent of \eqref{eq:weight} is  our underlying motivation for introducing an optimal numerical quadrature rule for computing expectations for Student's t distribution whose density is obtained from $\omega_{\upsilon}(x)$ for $\upsilon > 1$ by reparametrizing $\nu = \upsilon -1$, scaling $x = y/ \sqrt{\nu}$ and introducing an appropriate $\nu$-dependent normalization constant. Of particular interest is considering the Cauchy distribution,~i.e.,~essentially, Student's t distribution with $\nu = 1$, that is used as a sparsity inducing prior in inverse problems~\cite{Kaipio05}.

With Karvonen, we recently introduced in~\cite{SHK2024} the M\"obius-transformed trapezoidal rule  with optimal error convergence for computing integrals over the real line with monotonic Schwartz weights, i.e., weights that are positive Schwartz functions with monotonic convergence toward zero close to inifnity. In this paper, we extend the results on numerical integration in \cite{SHK2024} to polynomial weights and show that the  M\"obius-transformed trapezoidal rule still achieves an optimal rate of convergence.

To the best of our knowledge, numerical integration over the real line against polynomial weights has not been extensively studied, presumably due to their heavy tails. In particular, standard Gaussian quadrature cannot be constructed in such a setting because the associated orthogonal polynomials do not exist. Polynomial weights are, nonetheless, used in different areas of science, including statistics, finance~\cite{ZG2010}, and mathematical physics \cite{PhysRevE.54.R2197}.  

For general polynomial weights/decay, we are aware of a few results related to numerical integration over the real line for finitely smooth functions, namely, \cite{H1998,EGK2026} and \cite[Theorem~2,(ii)]{NS2023}. The former two papers consider the error of the discrete Fourier transform for approximating the Fourier transform on the real line, and the last paper considers numerical integration over $\R^d$. Although these papers assume somewhat different conditions, the analyzed numerical methods can be interpreted as a truncated trapezoidal rule in one dimension. Roughly speaking, given a smoothness parameter $\alpha$, they all prove 
convergence rates strictly slower than $n^{-\alpha}$ with $n$ sampling points. 

In contrast, we demonstrate that the M\"obius-transformed trapezoidal rule attains the convergence rate $n^{-\alpha}$ when the integrand function is in an appropriately weighted Sobolev space with smoothness index $\alpha \in \N$. Furthermore, we are able to generalize some of our results for fractional smoothness indices via complex interpolation of polynomially weighted Sobolev spaces~\cite{Lofstrom1982}. These results can again be related to \cite{H1998,EGK2026}, where the smoothness of a function is defined via the decay of its Fourier transform, thus allowing to consider fractional smoothness.  

This text is organized as follows. Section~\ref{sec:preliminaries} introduces and studies the considered class of weights and defines some associated weighted Sobolev spaces, including spaces of fractional smoothness defined via complex interpolation of Banach spaces. Section~\ref{sec:main_results} recalls the M\"obius-transformed trapezoidal rule and proves that it attains the optimal rate of convergence in the sense of worst-case error (for integer smoothness indices). Section~\ref{sec:exact_omega} demonstrates that the M\"obius-transformed trapezoidal rule exhibits polynomial exactness when applied to integrals involving the basic weight $\omega_\upsilon$ from \eqref{eq:weight0}. The predicted convergence rates are verified via numerical experiments in Section~\ref{sec:numerics}, and the concluding remarks are presented in the Section~\ref{sec:conclusion}.

\section{Preliminaries}
\label{sec:preliminaries}
In this section, we first formally introduce the studied weight functions and consider their basic properties. Next, we define the weighted and periodic Sobolev spaces used in our analysis. Finally, we recall the change of variables that forms the core of the studied quadrature rule and is motivated by a M\"obius transformation that maps the unit disk onto the lower half plane.

\subsection{Polynomial weights}

We start by defining the class of studied weights. In what follows, we assume all considered polynomials have real coefficients.

\begin{definition}
Let $q = q_{2m}: \R \to \R_+$ be a positive polynomial of degree $2m$ for $m \in \N$. We define the associated family of weights through
\begin{equation}
\label{eq:weight}
\rho_{\upsilon,q}(x)  = (q(x))^{-\upsilon / 2m}, \quad x \in \R,
\end{equation}
where $\upsilon \in \R$ is a free parameter.
\end{definition}

The basic weight $\omega_\upsilon$ from \eqref{eq:weight0} is obviously of the form \eqref{eq:weight}.
Take note that, in order for $q$ to be a positive polynomial, the coefficient of the highest order term in $q$ must be positive. Besides this simple observation, we do not further dwell on conditions that make a polynomial positive, but rather refer to \cite{Lasserre09} for more information.
In what follows, we drop the subindex $q$ from $\rho_{\upsilon,q}$,~i.e.,~we simply write $\rho_\upsilon$ and implicitly assume the existence of a fixed positive polynomial $q$ of degree $2m$ defining $\rho_\upsilon$ via \eqref{eq:weight}. The reader should note, however, that all constants in our estimates depend on $q$; in particular, the presented estimates are not uniform with respect to the choice of $q$.

The following lemma characterizes how differentiation affects the weights defined by \eqref{eq:weight} in comparison to $\omega_\upsilon$
\begin{lemma} \label{lemma:basic}
For any $\tau \in \N_{0}$ and $\upsilon \in \R$, 
\begin{equation*}
c_\upsilon \omega_\upsilon(x) \leq  \, \rho_{\upsilon}(x) \leq C_{\upsilon} \omega_\upsilon(x)
\qquad x \in \R, 
\end{equation*}
and
\begin{equation*}
\big|  \rho_\upsilon^{(\tau)}(x) \big| \leq C_{\upsilon, \tau} \, \omega_{\upsilon+\tau}(x), 
\qquad x \in \R, 
\end{equation*}
for some constants $c_\upsilon, C_{\upsilon},  C_{\upsilon, \tau}> 0$ that depend on $q$ in \eqref{eq:weight}.
\end{lemma}

\begin{proof}
The first bound follows from $\omega_\upsilon$ and  $\rho_{\upsilon}$ being smooth and positive functions that behave asymptotically as $|x|^{-\upsilon}$ at infinity.

A straightforward induction argument demonstrates that
\[
\rho_\upsilon^{(\tau)}(x) = \frac{r_{(2m -1)\tau}(x)}{(q_{2m}(x))^{(\upsilon + 2m \tau)/2m}}, \quad \tau \in \N_0,
\]
where $r_{(2m-1)\tau}$ is a polynomial of degree $(2m-1)\tau$. Since $\rho_\upsilon^{(\tau)}$ is a smooth function that behaves as $\mathcal{O}(|x |^{-\upsilon - \tau})$ at infinity, the claim follows.
\end{proof}

Observe that the second bound in Lemma~\ref{lemma:basic} is valid, yet suboptimal, even if $\upsilon$ is a negative integer multiple of $2m$ since in that case, $\rho_\upsilon^{(\tau)}$ vanishes after a finite number of differentiations. It is also essential to note that $\rho_{\upsilon}$ is {\em polynomially regular} as defined in \cite[p.~197]{Lofstrom1982} for any $\upsilon \in \R$.

\subsection{Weighted and periodic Sobolev spaces} 
Let $L_{\rho_\upsilon}^p (\R)$ denote the space of Lebesgue measurable functions on the real line with the norm
\[
\|f\|_{L_{\rho_\upsilon} ^p (\R)}
\coloneqq
\bigg(\int_\R |f(x)|^p \rho_\upsilon(x) \rd x\bigg)^{1/p}
, \qquad 1 \leq p < \infty.
\]
The corresponding $\rho_{\upsilon}$-weighted Sobolev space with a nonnegative smoothness index $\alpha \in \N_0$ is defined as
\begin{align} \label{eq:sobolev-space}
W^{\alpha,p}_{\rho_{\upsilon}}(\R) \coloneqq \bigg\{f\in L^p_{\rho_{\upsilon}}(\R)\ \Big| \ \|f\|_{W^{\alpha,p}_{\rho_{\upsilon}}(\R)} \coloneqq \bigg(\sum_{\tau=0}^\alpha \int_\R |f^{(\tau)}(x)|^p \rho_{\upsilon}(x) \rd x  \bigg)^{1/p}  \! \! < \infty \bigg\} 
\end{align}
for $1 \leq p  < \infty$. As discussed in \cite[Section~4]{Lofstrom1982}, $W^{\alpha,p}_{\rho_{\upsilon}}(\R)$ coincides with the weighted Bessel potential space
\begin{align} \label{eq:bessel-potential-space}
P^{\alpha,p}_{\rho_{\upsilon}}(\R) \coloneqq \Big\{f\in L_{\rho_{\upsilon}} ^p (\R) \ \big| \ \|f\|_{P^{\alpha,p}_{\rho_{\upsilon}}(\R)} \coloneqq \big \| \mathscr{F}^{-1} \big( (1 + | \cdot |^2)^{\alpha/2} \mathscr{F} f \big) \big\|_{L^p_{\rho_{\upsilon}}(\R)} < \infty \Big\} 
\end{align}
for $\alpha \in \N_0$, $1 < p < \infty$, and with $\mathscr{F}$ denoting the one-dimensional Fourier transform. In particular, $P^{\alpha,p}_{\rho_{\upsilon}}(\R)$ remains well-defined for all real $\alpha \geq 0$. Moreover, the complex interpolation space
\begin{equation}
\label{eq:complex_int_weight}
\big[P^{\alpha_0,p}_{\rho_{\upsilon_0}}(\R), P^{\alpha_1,p}_{\rho_{\upsilon_1}}(\R) ]_\theta =  P^{\tilde{\alpha},p}_{\rho_{\tilde{\upsilon}}}(\R), \qquad \theta \in (0,1),
\end{equation}
is also a weighted Bessel potential space with
\[
\tilde{\alpha} = (1-\theta)\alpha_0 + \theta \alpha_1, \quad 
\tilde{\upsilon} = (1-\theta)\upsilon_0 + \theta \upsilon_1,
\]
as revealed by \cite[Eq.~(5.3) in Theorem~4]{Lofstrom1982} since 
\[
\rho_{\upsilon_0}^{1-\theta} \rho_{\upsilon_1}^{\theta} = \rho_{\tilde{\upsilon}}.
\]
We define $W^{\alpha,p}_{\rho_{\upsilon}}(\R)$ in the customary manner  via \eqref{eq:sobolev-space} for $\alpha \in \N_0$ and $1 \leq p < \infty$ but slightly abuse the notation by extending the definition for a general $\alpha \geq 0$ and $1 < p < \infty$ through the identification $W^{\alpha,p}_{\rho_{\upsilon}}(\R) = P^{\alpha,p}_{\rho_{\upsilon}}(\R)$.

To deduce somewhat more optimal results on the convergence of the M\"obius-transformed trapezoidal rule for Sobolev spaces with integer smoothness order, we also need weighted Sobolev spaces that utilize different polynomial weights for different derivatives:
\begin{equation} 
\label{eq:sobolev-space2}
\widetilde{W}^{\alpha,p}_{\rho_{\upsilon}}(\R) \coloneqq \bigg\{f\in L^p_{\rho_{\upsilon+\alpha p}}(\R)\ \Big| \  \|f\|_{\widetilde{W}^{\alpha,p}_{\rho_{\upsilon}}(\R)} < \infty \bigg\}, \qquad 1 \leq p < \infty, \  \alpha \in \N_0,
\end{equation}
with
\[
\|f\|_{\widetilde{W}^{\alpha,p}_{\rho_{\upsilon}}(\R)} \coloneqq \bigg(\sum_{\tau=0}^\alpha \int_\R |f^{(\tau)}(x)|^p \rho_{\upsilon + p(\alpha - \tau)}(x) \rd x  \bigg)^{1/p}.
\]
Although the highest weak derivative of order $\alpha$ has the same weight in the definitions of $\| \, \cdot \, \|_{W^{\alpha,p}_{\rho_{\upsilon}}(\R)}$ and $\| \, \cdot \, \|_{\widetilde{W}^{\alpha,p}_{\rho_{\upsilon}}(\R)}$, the space $\widetilde{W}^{\alpha,p}_{\rho_{\upsilon}}(\R)$ is larger than $W^{\alpha,p}_{\rho_{\upsilon}}(\R)$ with a weaker norm due to the less strict condition on the behavior of the lower derivatives.\footnote{As a simple example, $f(x) = x$ belongs to $\widetilde{W}^{\alpha,1}_{\omega_{2+\epsilon-\alpha}}(\R)$ for any $\alpha \in \N_0$ and $\epsilon > 0$, but not even to $L_{\omega_2}^1 (\R) = W^{0,1}_{\omega_{2}}(\R)$.} To be more precise, one has
\begin{equation}
\label{eq:embeddings}
W^{\alpha,p}_{\rho_{\upsilon}}(\R) \subset \widetilde{W}^{\alpha,p}_{\rho_{\upsilon}}(\R) \subset W^{\alpha,p}_{\rho_{\upsilon + p \alpha}}(\R),
\end{equation}
which follows from the decreasing strictness of the integrability conditions  in this chain of spaces. We do not consider fractional smoothness orders for the spaces \eqref{eq:sobolev-space2} because we are not aware of an explicit characterization of the related interpolation spaces.

The convergence of the M\"obius-transformed trapezoidal rule is based on properties of periodic Sobolev spaces on the one-dimensional torus $\T$ (i.e.,~the unit circle) defined by identifying the end points of the interval $[0,2 \pi]$. For smoothness indices $\alpha \in \N_0$, such spaces are traditionally defined as
\begin{align}\label{eq:Lq-per-sob} 
    W^{\alpha,p}(\T) &= \bigg\{ f\in L^p(\T)  \ \Big| \  \|f\|_{W^{\alpha,p}(\T)}^p \coloneqq 
    \sum_{\tau=0}^{\alpha} \int_{\T} |f^{(\tau)}(x) |^p \rd x   < \infty \bigg\} 
    \end{align}
    for $1\leq p<\infty$. Via a linear isometry, $W^{\alpha,p}(\T)$ has an alternative characterization as a subspace of the standard Sobolev space $W^{\alpha,p}(0,2\pi)$ on $(0,2\pi)$ with compatibility conditions between the values of weak derivatives at $0$ and $2\pi$; see,~e.g.,~\cite[Proposition~2.3]{SHK2024} that also holds for $p=1$ since $W^{1,1}(0,2\pi) \subset C([0,2\pi])$ with a continuous embedding~\cite[Chapter~8, Remark~10]{Brezis11}. That is, $W^{\alpha,p}(\T)$ can be identified with
    \begin{align}
    \label{eq:Lq-per-sob2} 
    \nonumber
    W^{\alpha,p}_{\rm per}(0,2\pi)&:=
    \bigg\{ f\in L^p(0,2\pi)  \ \Big| \  \|f\|_{W^{\alpha,p}(0,2\pi)}^p\coloneqq \sum_{\tau=0}^{\alpha} \int_0^{2\pi} |f^{(\tau)}(\theta) |^p \rd \theta  < \infty, \\[-1mm]
    & \kern4.2cm f^{(\tau)}(0)=f^{(\tau)}(2\pi) \text{ for } \tau=0,\ldots,\alpha-1  \bigg\} 
\end{align}
for $\alpha \in \N_0$. According to, e.g., \cite[Section~1.44]{Ullrich2007},  the space $W^{\alpha,p}(\T)$ allows yet another equivalent definition as
\begin{align}\label{eq:Lq-per-sob3} 
    W^{\alpha,p}(\T) &\coloneqq \Bigg\{ f\in L^p(\T)  \ \Big| \  \|f\|_{W^{\alpha,p}(\T)} \coloneqq \bigg\| \sum_{k= -\infty}^\infty \hat{f}(k) (1 + k^2)^{\alpha/2} \, \mathrm{e}^{{\rm i} k \, \cdot} \bigg \|_{L^p(\T)}  \!\! \! \! < \infty \Bigg\}    
\end{align}
for  $1  < p < \infty$ and with $\hat{f}(k)$, $k \in \Z$, denoting the Fourier coefficients of $f$. As in the case of \eqref{eq:bessel-potential-space} for weighted Bessel potential spaces, the definition \eqref{eq:Lq-per-sob3} can also be used for real smoothness indices $\alpha \geq 0$, and we thus use \eqref{eq:Lq-per-sob3} to extend the definition of $W^{\alpha,p}(\T)$ and its norm for $1 < p < \infty$ and $\alpha \geq 0$. Finally, with this convention, complex interpolation of function spaces gives
\begin{equation}
\label{eq:complex_int_per}
\big[ W^{\alpha_0,p}(\T), W^{\alpha_1,p}(\T)\big]_\theta = 
W^{\tilde{\alpha},p}(\T), \qquad 1  < p < \infty,
\end{equation}
where $\tilde{\alpha} = (1-\theta)\alpha_0 + \theta \alpha_1$; see,~e.g.,~\cite[Theorem~1.7]{Ullrich2007} with $r=\tilde{\alpha}$, $p_0 = p_1$ and $q_0 = q_1 = 2$.

\subsection{M\"obius transformation induced change of variables}

The construction of the M\"obius-transformed trapezoidal rule is motivated by conformally mapping the unit disk in the complex plain onto a half plane and thus the unit circle onto a line.
Without significant loss of generality (cf.~\cite{Hille59,SHK2024}), consider
\begin{equation}
    \Phi_{\gamma}(z) = - {\rm i} \gamma \, \frac{ z +1 }{z - 1}, \qquad z \in \C, \ \gamma \in \R_+,
\end{equation}
that maps the unit disk onto the lower half plane. Parametrizing the restriction of $\Phi_{\gamma}$ onto the unit circle by the polar angle $\theta\in (0,2\pi)$ gives the sought for change of variables,
\begin{align}
\label{eq:variable_change}
\phi_\gamma(\theta) &:= \Phi_{\gamma}(\re^{\ri \theta})
= - \gamma \cot\!\Big(\frac{\theta}{2}\Big).
\end{align}
In the subsequent analysis, we also need the identities
\begin{equation}
\label{eq:invphi}
    \phi_\gamma'(\theta) = \frac{\gamma}{2 \sin^2(\theta/2)}, \quad 
    \phi_\gamma^{-1}(x) = 2 \, {\rm arccot}\!\left(-\dfrac{x}{\gamma}\right), \quad (\phi_\gamma^{-1})'(x) = \frac{2 \, \gamma}{\gamma^2 + x^2}.
\end{equation}
For the sake of brevity, we usually write $\phi(\theta)$ instead of $\phi_\gamma(\theta)$.

\section{M\"obius-transformed trapezoidal rule for polynomial weights}
\label{sec:main_results}

Our aim is to approximate the weighted integral
\begin{align}
\label{eq:integral}
    I_{\rho_\upsilon}(f) &\coloneqq \int_{\R} f(x) \rho_{\upsilon}(x)\rd x = \int_0^{2\pi} f(\phi(\theta)) \rho_{\upsilon}(\phi(\theta)) \phi'(\theta) \rd \theta
\end{align}
for a continuous function $f: \R \to \C$. The studied approximation,~i.e.,~the M{\"o}bius-transformed trapezoidal rule, is defined as
\begin{align}\label{eq:quadrature}
    Q_{\rho_\upsilon,n}(f) \coloneqq \frac{2\pi}{n}\sum_{j=1}^n f(\phi(\theta_j))\rho_\upsilon (\phi(\theta_j)) \phi'(\theta_j) \approx I_{\rho_\upsilon}(f) ,
\end{align}
where $\theta_j\coloneqq 2\pi j/n - \pi/n$ for $j=1,\ldots,n$.\footnote{The purpose of the shift by $\pi/n$ is to avoid the end points of the interval $[0,2\pi]$ that correspond to infinity on the real line. See also \cite[Remark~3.3]{SHK2024}.}
That is, one simply applies the trapezoidal rule on the unit circle to the right-hand side of the identity~\eqref{eq:integral}.
In what follows, when an integrand function $f$ lives in a Sobolev space $W^{\alpha,p}_{\rm loc}(\R)$ with $\alpha>1/p$, we always consider the continuous representative in the corresponding Sobolev equivalence class. This makes the application of the M{\"o}bius-transformed trapezoidal rule $Q_{\rho_{\upsilon},n}(f)$ well-defined.

\subsection{Upper bound of convergence for $\alpha \in \N$}
\label{sec:upperbound}
We start by presenting our main result on the convergence of the M\"obius-transformed trapezoidal rule for a positive integer smoothness index $\alpha$. The proof is completed by Proposition~\ref{thm:f-rho-norm-p} and Corollary~\ref{cor:integer} below, and it is based on a classical result on the convergence of the standard trapezoidal rule on periodic Sobolev spaces.
\begin{theorem}
\label{thm:main_result}
    For any $\upsilon \in \R$, $1 < p < \infty$ and $\alpha \in \N$,
        \begin{equation}
            \label{eq:main_result}
        \big| I_{\rho_\upsilon}(f) - Q_{\rho_\upsilon,n}(f) \big| \leq C_{\upsilon, \alpha, p} \, n^{-\alpha} \| f \|_{\widetilde{W}^{\alpha,p}_{\rho_{p(\upsilon - 2  - 2 \alpha) + 2}}(\R)},  
    \end{equation}
    where $C_{\upsilon, \alpha, p} > 0$ is independent of $f$ and $n \in \N$.
\end{theorem}

\begin{proof}
The assertion follows by combining Corollary~\ref{cor:integer} with \cite[Theorem~2.4.1]{T2018_book}, which considers the convergence of the trapezoidal rule on periodic Sobolev spaces.
\end{proof}

Although Theorem~\ref{thm:main_result} is formulated for the same weight family appearing on both sides of \eqref{eq:main_result}, one could, in fact, use weights defined by different positive polynomials on different sides of \eqref{eq:main_result} by virtue of Lemma~\ref{lemma:basic} that guarantees the associated weighted Sobolev norms are equivalent. In particular, one is allowed to replace $\rho_{p(\upsilon - 2  - 2 \alpha) + 2}$ by $\omega_{p(\upsilon - 2  - 2 \alpha) + 2}$ on the right-hand side of \eqref{eq:main_result}. The same conclusion also applies to Theorem~\ref{thm:main_result2} below.

Sometimes, it is natural to assume that the integrand $f$ belongs  to a $\rho_\upsilon$-weighted Sobolev space,~i.e.,~that the weight appearing in the approximated integral and the one in the definition of the domain for the integrand function are the same. Assuming the smoothness of $f$ is not a restrictive factor, the following corollary gives the best convergence rate as a function of $\upsilon$ and $p$.

\begin{corollary}
\label{cor:main_result}
    Let $\upsilon \in \R$ and $1 < p < \infty$. For $\N \ni \alpha \leq  (p-1)(\upsilon - 2)/(2p)$,
    \begin{equation}
    \label{eq:main_result_upsilon}
        \big| I_{\rho_\upsilon}(f) - Q_{\rho_\upsilon,n}(f) \big| \leq C_{\upsilon, p} \, n^{-\alpha} \| f \|_{\widetilde{W}^{\alpha,p}_{\rho_{\upsilon}}(\R)},  
    \end{equation}
    where $C_{\upsilon, p} > 0$ is independent of $f$ and $n \in \N$.
\end{corollary}

\begin{proof}
The assertion follows from Theorem~\ref{thm:main_result} by equating $\upsilon = p(\upsilon - 2  - 2\beta) + 2$, solving for $\beta$, and choosing $\alpha$ to be a positive integer that is smaller or equal to~$\beta$.
\end{proof}

Note that one can find a positive integer $\alpha \leq  (p-1)(\upsilon - 2)/(2p)$ if and only if $\upsilon \geq (4 p -2)/(p-1)$.
The lower bound for $\upsilon$ decreases and the highest attainable convergence rate $\alpha$ increases as the function of $p$, with the respective limits $4$ and $(\upsilon -2)/2$ as $p \to \infty$. Such behavior is natural since for a fixed $\upsilon > 1$, a higher $p$ reduces the speed with which elements of $W^{\alpha,p}_{\rho_{\upsilon}}(\R)$ may grow at infinity. 

Let us then complete the proof of Theorem~\ref{thm:main_result}. We start by showing that the integrand on the right-hand side of \eqref{eq:integral} belongs to a periodic Sobolev space whose smoothness is dictated by the regularity and behavior at infinity of $f$. 

\begin{proposition}
    \label{thm:f-rho-norm-p}
    Let $f\in \widetilde{W}^{\alpha,p}_{\rho_{p(\upsilon - 2  - 2\alpha)+2}}(\R)$ with some $\alpha \in \N_0$ and $1 < p < \infty$, and define $g_{\upsilon} = ((f\rho_\upsilon)\circ \phi) \phi'$. It holds that
    \begin{equation}
    \label{eq:g-tau}
    \big\|g_{\upsilon}^{(\tau)} \big\|_{L^p(0,2\pi)} \leq C_{\upsilon, \alpha,p} \|f\|_{\widetilde{W}^{\alpha,p}_{\rho_{p(\upsilon - 2  - 2\alpha) + 2}}(\R)}, \qquad \tau = 0, \dots, \alpha,
    \end{equation}
    where $C_{\upsilon, \alpha,p} > 0$ is independent of $f$. Moreover, for $\tau=0,\ldots,\alpha-1$,
    \begin{equation}\label{eq:g-boundary_}
    \lim_{\theta\to0^+} g_{\upsilon}^{(\tau)}(\theta) =
    \lim_{\theta\to 2\pi ^-} g_{\upsilon}^{(\tau)}(\theta) = 0.
    \end{equation}
    As a consequence, $g_{\upsilon} \in W^{\alpha,p}_{\rm per}(0,2\pi)$.
\end{proposition}
\begin{proof}
The following two results are proven in \cite[Appendix~A]{SHK2024}: (i) The weak derivative $g_\upsilon^{(\tau)}$, $\tau \in \N_0$, is a finite linear combination of terms of the form
\begin{equation}
\label{eq:ind1}
    \big((f^{(\tau_1)} \rho_\upsilon^{(\tau_2)}) \circ \phi \big) \prod_{j=1}^{\tau_1 + \tau_2+1} \phi^{(\tau_{3,j})},
\end{equation}
    where the nonnegative integers $\tau_1$, $\tau_2$, and $\tau_{3,j}$ satisfy 
    \[
    \tau_1 + \tau_2 \leq \tau \quad \text{and} \quad \sum_{j=1}^{\tau_1 + \tau_2 +1} \tau_{3,j} = \tau+1.
    \]
(ii) It holds that
\begin{equation}
\label{eq:ind2}
    \phi^{(\tau)}(\theta) = \frac{\psi_\tau(\theta)}{\sin^{\tau+1}(\theta/2)}, \qquad \tau \in \N_0,
\end{equation}
where $\psi_\tau \in C^\infty(\R)$ is a bounded finite linear combination of products of trigonometric functions.

Let us first tackle \eqref{eq:g-tau}. Combining \eqref{eq:ind1} and \eqref{eq:ind2}, it follows from the triangle inequality that it is enough to prove \eqref{eq:g-tau} with  
   \begin{equation}
   \label{eq:widetildeg}
    \widetilde{g}_{\upsilon,\tau_1,\tau_2,\eta}(\theta) = \frac{(f^{(\tau_1)} \rho_\upsilon^{(\tau_2)}) \circ \phi(\theta)}{\sin^{\eta+2}(\theta /2)}, \qquad \tau_1 + \tau_2 \leq \tau \leq \alpha, \ \ \eta = \tau + \tau_1 + \tau_2 \leq 2\tau,
    \end{equation}
replacing $g_\upsilon^{(\tau)}$. A direct calculation gives,
\begin{align}
\label{eq:tilde_g_p}
\big\|\widetilde{g}_{\upsilon,\tau_1,\tau_2,\eta}(\theta) \big\|_{L^p(0,2 \pi)}^p &= 
    \int_0 ^{2\pi} \left|f^{(\tau_1)}(\phi(\theta)) \, \rho_\upsilon^{(\tau_2)}(\phi(\theta)) \, \frac{1}{\sin^{\eta+2}(\theta/2)}\right|^p \rd \theta \nonumber \\[1mm]
    & =
  \int_{\R} \left|f^{(\tau_1)}(x) \, \rho_\upsilon^{(\tau_2)}(x) \, \frac{1}{\sin^{\eta + 2 }(\phi^{-1}(x)/2)} \right|^p \big|(\phi^{-1} ) ' (x)\big| \rd x
   \nonumber  \\[1mm]
    &\le 
    \int_{\R} \big|f^{(\tau_1)}(x) \big|^p \rho_{p(\upsilon - 2  - 2 \alpha) + 2 + p (\alpha - \tau_1)}(x) \rd x \  \times  \nonumber \\
    & \quad \ \times  \  \sup_{x\in\R}  \left|\frac{(\phi^{-1})'(x)}{\sin^{p(\eta+2)}(\phi^{-1}(x)/2)} \frac{(\rho_\upsilon^{(\tau_2)}(x))^p}{\rho_{p(\upsilon - 2  - 2 \alpha) + 2 + p (\alpha - \tau_1)}(x)}\right|  \nonumber \\[1mm]
    &\le 
     \|f\|_{\widetilde{W}^{\alpha,p}_{\rho_{p(\upsilon - 2  - 2\alpha) + 2}}(\R)}^p \  \times \\[1mm]
    & \quad  \ \times  \  \sup_{x\in\R}  \left|\frac{(\phi^{-1})'(x)}{\sin^{p(\eta+2)}(\phi^{-1}(x)/2)} \frac{(\rho_\upsilon^{(\tau_2)}(x))^p}{\rho_{p(\upsilon - 2  - \alpha - \tau_1) + 2}(x)}\right|. \nonumber
\end{align}
The estimate \eqref{eq:g-tau} thus follows by showing that the supremum term on the right-hand side of \eqref{eq:tilde_g_p} is finite. 

By virtue of Lemma~\ref{lemma:basic}, 
\begin{align}
\label{eq:rho_ratio_}
\left| \frac{(\rho_\upsilon^{(\tau_2)}(x))^p}{\rho_{p(\upsilon - 2  - \alpha - \tau_1) + 2}(x)} \right|
&\leq C \big(1 + x^2\big)^{-p(\upsilon + \tau_2) /2 + p(\upsilon-2-\alpha - \tau_1)/2+1} \nonumber \\
&= C \big(1 + x^2\big)^{ - p( \alpha + \tau_1 + \tau_2 + 2)/2 +1} \nonumber \\[1mm]
&\leq C' \big(1 + x^2\big)^{ - p( \eta + 2)/2 +1},
\end{align}
where the last step follows from the conditions in \eqref{eq:widetildeg} and the constants depend on $\upsilon$, $\alpha$, $\tau_1$, and $\tau_2$.
On the other hand, by basic trigonometry including \eqref{eq:invphi},
\begin{align}
\label{eq:inv_sin_}
\frac{(\phi^{-1})'(x)}{\sin^{p(\eta+2)}(\phi^{-1}(x)/2)} &= \frac{2 \gamma}{\gamma^2 + x^2} \frac{1}{\sin^{p(\eta + 2)}({\rm arccot}(-x/\gamma))} \nonumber \\[1mm]
&= \frac{2}{\gamma^{p\eta + 2p - 1}} \big(\gamma^2 + x^2 \big)^{p(\eta+2)/2-1}.
\end{align}
Combining \eqref{eq:tilde_g_p}, \eqref{eq:rho_ratio_} and \eqref{eq:inv_sin_} proves \eqref{eq:g-tau} and, in particular, that $g_\upsilon \in W^{\alpha, p}(0,2 \pi)$.

Let us define
\[
h_{\upsilon, \tau, \varepsilon}(\theta) = \frac{g_{\upsilon}^{(\tau)}(\theta)} {\sin^{\varepsilon}(\theta/2)}, \qquad \theta \in (0, 2 \pi),
\]
for a fixed but arbitrary $\varepsilon > 0$. To prove \eqref{eq:g-boundary_}, it is enough to show that $h_{\upsilon, \tau, \varepsilon}$ is bounded on the interval $(0,2\pi)$ for $\tau = 0,\dots, \alpha-1$. By the Sobolev embedding theorem in one dimension (e.g.,~\cite[Theorem~8.8]{Brezis11}),
\[
\|v \|_{L^\infty(0,2 \pi)} \leq C \|v \|_{W^{1,1} (0,2\pi)},
\]
which means that it is sufficient to deduce bounds for the $L^{1} (0,2\pi)$ norms of $h_{\upsilon, \tau, \varepsilon}$ and
$$
h'_{\upsilon,\tau,\varepsilon}(\theta) = \frac{g_{\upsilon}^{(\tau+1)}(\theta)} {\sin^{\varepsilon}(\theta/2)} - \frac{ \varepsilon \cos(\theta/2) \, g_{\upsilon}^{(\tau)}(\theta)} {2 \sin^{\varepsilon+1}(\theta/2)}, \qquad  \tau = 0,\dots, \alpha-1.
$$ 
Hence, combining \eqref{eq:ind1}, \eqref{eq:ind2} and the triangle inequality demonstrates that it is, in fact, enough to show that 
 \[
    \widetilde{g}_{\upsilon,\tau_1,\tau_2,\eta+ \varepsilon +1 - \delta(\tau)}(\theta) = \frac{(f^{(\tau_1)} \rho_\upsilon^{(\tau_2)}) \circ \phi(\theta)}{\sin^{\eta+\varepsilon + 3  - \delta(\tau)}(\theta /2)} 
    \]
belongs to $L^{1} (0,2\pi)$ for $\tau$, $\tau_1$,  $\tau_2$ and  $\eta$ as defined in~\eqref{eq:widetildeg}. Here, $\delta(\tau) = 1$ if $\tau = \alpha$ and $\delta(\tau) = 0$ otherwise.

By estimating as in \eqref{eq:tilde_g_p},
\begin{align}
\label{eq:limit_bound}
\big\|& \widetilde{g}_{\upsilon,\tau_1,\tau_2,\eta+ \varepsilon + 1 - \delta(\tau)}  \big\|_{L^1(0,2\pi)} \\[1mm]
    & \qquad = \int_{\R} \bigg|f^{(\tau_1)}(x) \, \rho_\upsilon^{(\tau_2)}(x) \, \frac{1}{\sin^{\eta +\varepsilon + 3 - \delta(\tau)}(\phi^{-1}(x)/2)} \bigg| \big|(\phi^{-1} ) ' (x)\big| \rd x \nonumber \\[1mm]
  & \qquad \leq \left(\int_{\R} \big|f^{(\tau_1)}(x) \big|^p \rho_{p(\upsilon - 2  - 2\alpha) + 2 + p(\alpha - \tau_1)}(x) \rd x \! \right)^{1/p} \! \! \! \|w\|_{L^{p'}(\R)},
\end{align}
where the second step follows from H\"older's inequality, with $p' = p/(p-1)$ denoting the H\"older conjugate of $p$ and
\[
w(x) = 
\left| \frac{(\phi^{-1} )'(x) }{\sin^{\eta + \varepsilon + 3 -\delta(\tau)}(\phi^{-1}(x)/2)} \right| 
   \left| \frac{(\rho_\upsilon^{(\tau_2)}(x))}{(\rho_{p(\upsilon - 2  - \alpha - \tau_1)+2}(x))^{1/p}} \right|, \qquad x \in \R.
\]
Since the first integral on the right-hand side of the inequality \eqref{eq:limit_bound} is bounded by 
$\|f\|_{\widetilde{W}^{\alpha,p}_{\rho_{p(\upsilon - 2  - 2\alpha) + 2}}(\R)}$, the proof can be completed by showing that $w \in L^{p'}(\R)$.

Through the same argument as in \eqref{eq:rho_ratio_}, we obtain
\begin{align*}
 \left| \frac{(\rho_\upsilon^{(\tau_2)}(x))}{(\rho_{ p(\upsilon - 2  - \alpha - \tau_1)+2}(x))^{1/p}} \right|
 &\le
 C_{\upsilon,\alpha,\tau} \big(1 + x^2\big)^{ -( \alpha + \tau_1 + \tau_2 + 2)/2 +1/p} \\[0mm]
 & =  C_{\upsilon,\alpha,\tau} \big(1 + x^2\big)^{ -( \tau  + \tau_1 + \tau_2 + 2 + (\alpha - \tau))/2 +1/p} \\[1mm]
 & \le C_{\upsilon,\alpha,\tau} \big(1 + x^2\big)^{ -( \eta + 2 + (1- \delta(\tau)))/2 +1/p},
\end{align*}
and as in \eqref{eq:inv_sin_}, we deduce that 
\begin{align*}
\left |\frac{(\phi^{-1})'(x)}{\sin^{\eta+\varepsilon +3 - \delta(\tau)}(\phi^{-1}(x)/2)} \right| &= \frac{2 \gamma}{(\gamma^2 + x^2)} \frac{1}{\sin^{\eta+\varepsilon +3 - \delta(\tau)}({\rm arccot}(-x/\gamma))} \nonumber \\[1mm]
&\leq C \big(\gamma^2 + x^2 \big)^{(\eta+\varepsilon +1 - \delta(\tau))/2} .
\end{align*}
It thus holds 
\[
|w(x)|\leq C \big(1 + x^2\big)^{ -1  + 1/p + \varepsilon/2} = C \big(1 + x^2\big)^{ -1/p'  + \varepsilon/2}
\]
from which the claim follows by fixing $\varepsilon \in (0, 1/p')$.
\end{proof}

Let us define a linear operator
\begin{equation}
\label{eq:Tnu}
T_\upsilon: f \mapsto g_\upsilon = ((f\rho_{\upsilon})\circ \phi) \phi', \qquad \upsilon \in \R,
\end{equation}
that maps an integrand function living on the real line to the transformed weighted integrand on the right-hand side of \eqref{eq:integral}.
The following simple corollary considers the boundedness of $T_\upsilon$ and completes the proof of Theorem~\ref{thm:main_result}.

\begin{corollary}
\label{cor:integer}
    The operator $T_{\upsilon}$, $\upsilon \in \R$, is bounded from $\widetilde{W}^{\alpha,p}_{\rho_{p (\upsilon - 2 + 2\alpha) + 2}}(\R)$ to  $W^{\alpha,p}(\T)$ for all $\alpha \in \N_0$ and $1 < p < \infty$.
\end{corollary}

\begin{proof}
The claim follows from Proposition~\ref{thm:f-rho-norm-p}, the definition of $W^{\alpha,p}_{\rm per}(0,2\pi)$ in \eqref{eq:Lq-per-sob2}, and the identification of $W^{\alpha,p}_{\rm per}(0,2\pi)$ with  $W^{\alpha,p}(\T)$; see,~e.g.,~\cite[Proposition~2.3]{SHK2024}.
\end{proof}

A remark on the case $p=1$, which is excluded from Theorem~\ref{thm:main_result}, completes this section.

\begin{remark}
\label{remark:case_p1}
As Corollary~\ref{cor:integer} is a direct consequence of Proposition~\ref{thm:f-rho-norm-p} (cf.~the discussion before \eqref{eq:Lq-per-sob2}) and the approximation result of \cite[Theorem~2.4.1]{T2018_book} also holds for $p=1$, the validity of Theorem~\ref{thm:main_result} for $p=1$ only hinges on Proposition~\ref{thm:f-rho-norm-p}. The presented proof of \eqref{eq:g-tau} extends to the case $p=1$, but that of \eqref{eq:g-boundary_} fails: for $p' = \infty$, there exists no suitable $\varepsilon > 0$ on the last line of the proof. However, making the norm of the Sobolev space where $f$ lives stronger by decreasing the index of the associated weight allows to prove \eqref{eq:g-boundary_} for $p=1$ as well. Hence, for any $\alpha \in \N$ and $\delta > 0$,
    \begin{equation}
            \label{eq:main_result_p}
        \big| I_{\rho_\upsilon}(f) - Q_{\rho_\upsilon,n}(f) \big| \leq C_{\upsilon, \alpha, \delta} \, n^{-\alpha} \| f \|_{\widetilde{W}^{\alpha,1}_{\rho_{\upsilon - 2 \alpha - \delta}}(\R)},  
    \end{equation}
    where $C_{\upsilon, \alpha, \delta} > 0$ is independent of $f$ and $n \in \N$.
\end{remark}

\subsection{Upper bound of convergence for $\alpha > 1/p$}
\label{sec:upperbound_interp}
Next, we formulate a variant of Theorem~\ref{thm:main_result} for fractional order spaces. This result is suboptimal compared to Theorem~\ref{thm:main_result}: since we do not have a characterization of the spaces $\widetilde{W}^{\alpha,q}_{\rho_{\upsilon}}(\R)$ for all $\alpha > 1/p$ at our disposal, we present the result for the larger standard weighted spaces $W^{\alpha,q}_{\rho_{\upsilon}}(\R)$.
\begin{theorem}
\label{thm:main_result2}
    For any $\upsilon \in \R$, $1 < p < \infty$ and $\alpha > 1/p$,
        \begin{equation}
            \label{eq:main_result2}
        \big| I_{\rho_\upsilon}(f) - Q_{\rho_\upsilon,n}(f) \big| \leq C_{\upsilon, \alpha, p} \, n^{-\alpha} \| f \|_{W^{\alpha,p}_{\rho_{p(\upsilon - 2  - 2\alpha)+2}}(\R)},  
    \end{equation}
    where $C_{\upsilon, \alpha, p} > 0$ is independent of $f$ and $n \in \N$.
\end{theorem}

\begin{proof}
The claim follows by combining Proposition~\ref{cor:interp} presented below with \cite[Theorem~2.4.1]{T2018_book}.
\end{proof}

Via complex interpolation of Sobolev spaces, we deduce the following weaker analogue of Corollary~\ref{cor:integer} for $\alpha \geq 0$ and $1 < p < \infty$.

\begin{proposition}
\label{cor:interp}
    The operator $T_{\upsilon}$, $\upsilon \in \R$, is bounded from $W^{\alpha,p}_{\rho_{p (\upsilon - 2 - 2 \alpha) + 2}}(\R)$ to $W^{\alpha,p}(\T)$ for all real $\alpha \geq 0$ and $1 < p < \infty$.
\end{proposition}

\begin{proof}
By virtue of \eqref{eq:complex_int_weight} and \eqref{eq:complex_int_per},
\begin{equation*}
\big[W^{\alpha,p}_{\rho_{p (\upsilon - 2 - 2 \alpha) + 2}}(\R), W^{\alpha+1,p}_{\rho_{p (\upsilon - 2 - 2 (\alpha + 1)) + 2}}(\R)\big]_\theta = W^{\alpha+\theta,p}_{\rho_{p (\upsilon - 2 - 2 (\alpha + \theta)) + 2}} (\R)
\end{equation*}
and 
\begin{equation*}
\big[ W^{\alpha,p}(\T), W^{\alpha+1,p}(\T) \big]_\theta = W^{\alpha+\theta,p}(\T)
\end{equation*}
for any $1 < p < \infty$, $\alpha \in \N_0$ and $0 < \theta < 1$. As $W^{\alpha,p}_{\rho_{p (\upsilon - 2 - 2 \alpha) + 2}}(\R) \subset \widetilde{W}^{\alpha,p}_{p (\upsilon - 2 - 2 \alpha) + 2}(\R)$ with a stronger norm for $\alpha \in \N_0$, Corollary~\ref{cor:integer} and basic theory on complex interpolation~\cite{Bergh1976} demonstrate that $T_{\upsilon}$, defined by \eqref{eq:Tnu}, is bounded from $W^{\alpha+\theta,p}_{\rho_{p (\upsilon - 2 - 2 (\alpha + \theta)) + 2}}(\R)$ to $W^{\alpha+\theta,p}(\T)$. This proves the claim via varying $\alpha \in \N_0$ and $0 < \theta < 1$.
\end{proof}

We complete this section with a remark that interprets Theorems~\ref{thm:main_result} and \ref{thm:main_result2} in case one has freedom to choose how the whole integrand on the left-hand side of \eqref{eq:integral} is divided into the ``target function'' and the ``weight''.

\begin{remark}
\label{remark:cancel}
    A given fixed integrand function $h:=f\rho_{\upsilon}$ on the left-hand side of \eqref{eq:integral} can be factored into $f$ and $\rho_{\upsilon}$ in an arbitrary manner as long as $\rho_{\upsilon}$ remains a polynomial weight of the form~\eqref{eq:weight}. Theorems~\ref{thm:main_result} and \ref{thm:main_result2} are thus valid for the best possible $\alpha$ obtained by optimizing $p$ and $\upsilon$ through the selection of the pair $(f,\rho_{\upsilon})$. 

Moreover, as the proof of Proposition~\ref{thm:f-rho-norm-p} tackles different terms in $g_\upsilon^{(\tau)}$ separately via the triangle inequality, it does not account for possible cancellation between them. Hence, one may sometimes observe a much higher convergence rate than predicted by Theorems~\ref{thm:main_result} and \ref{thm:main_result2}, even for $\alpha \in \N$ covered by the more optimal Theorem~\ref{thm:main_result}. This phenomenon will be demonstrated by one of the numerical examples in Section~\ref{sec:numerics}.
\end{remark}

\subsection{General lower bound for $\alpha \in \N$}\label{sec:lowerbound}

In this subsection, we discuss the optimality of our method, when measured by the worst-case error. 
To this end, let $A_{\rho_\upsilon, n}$ be a general linear $n$-point quadrature rule for evaluating the integral \eqref{eq:integral},
\begin{equation}
\label{eq:A_quadrature}
 A_{\rho_\upsilon, n}(f):=\sum_{j=1}^n w_j f(x_j),
\end{equation}
where $\{w_j\}_{j=1}^n \subset \R$ and $ \{x_j\}_{j=1}^n \subset \R$ are the associated weights and nodes, respectively. Note that choosing $\beta = p(\upsilon-2-2\alpha)+2$ in the following proposition proves the worst-case optimality of the M\"obius-transformed trapezoidal rule for $1 < p < \infty$ and $\alpha \in \N$, cf.~Theorem~\ref{thm:main_result}.

\begin{proposition}
\label{prop:lower_bound}
Let $1 < p< \infty$, $\upsilon \in \R$,
and $\alpha \in \N$. 
For any fixed linear quadrature of the form \eqref{eq:A_quadrature} and $\beta \in \R$,
    \begin{equation}\label{eq:lower-bound}
        \sup_{0\ne f\in \widetilde{W}^{\alpha,p}_{\rho_{\beta}}(\R)}
        \frac{|I_{\rho_\upsilon}(f)-A_{\rho_\upsilon,n}(f)|}{\|f\|_{\widetilde{W}^{\alpha,p}_{\rho_{\beta}}(\R)}} \ge c \, n^{-\alpha},
    \end{equation}
    where $c>0$ is a constant independent of $n$.
\end{proposition}
\begin{proof}
    Denote by $\mathring{W}^{\alpha,p}(0,1)$ the closure of the compactly supported test functions $C^\infty_{\rm c}(0,1)$ in the topology of $W^{\alpha,p}(0,1)$. According to basic theory of Sobolev spaces~\cite[Lemma~3.22]{Adams1975}, each function in $\mathring{W}^{\alpha,p}(0,1)$ can be identified with an element of $\widetilde{W}^{\alpha,p}_{\rho_{\beta}}(\R)$ via zero-continuation. 
    In particular, for $f\in\mathring{W}^{\alpha,p}(0,1)$,
    \begin{align}
    \label{eq:norm_equivalence}
        \|f\|_{\widetilde{W}^{\alpha,p}_{\rho_{\beta}}(\R)}
    &=\bigg(\sum_{\tau=0}^{\alpha}\int_0^1 |f^{(\tau)}(x)|^p \rho_{\beta+p(\alpha-\tau)}(x)\rd \nonumber x\bigg)^{1/p}
    \\[1mm]
    & \le C_{\beta, \alpha} \, \|f\|_{\mathring{W}^{\alpha,p}(0,1)},
    \end{align}
    where the zero-extension of $f$ is denoted by the same symbol and
    \[
    C_{\beta, \alpha} := \max_{\tau = 0, \dots, \alpha} \, \max_{x\in[0,1]} (\rho_{\beta + p(\alpha - \tau)}(x))^{1/p}.
    \]
     Moreover, according to \cite[Remark~2]{T1990}, 
    \begin{equation}
    \label{eq:lowerbound0}
    \sup_{0\ne g\in \mathring{W}^{\alpha,p}(0,1)}
        \frac{\bigg |{\displaystyle \int_0^1} g(x) \rd x-\widehat{A}_n(g) \bigg|}{\|g\|_{ \mathring{W}^{\alpha,p}(0,1)}} \ge c' n^{-\alpha},
    \end{equation}
    for any $n$-point linear quadrature rule $\widehat{A}_n$ on $(0,1)$, with $c'$ independent of~$n$. Finally, note that the mapping $R_\upsilon: h \mapsto \rho_\upsilon h$ is a linear isomorphism on $\mathring{W}^{\alpha,p}(0,1)$ due to the smoothness and positiveness of $\rho_\upsilon$.
    
  Let us then prove \eqref{eq:lower-bound}. First of all, by virtue of \eqref{eq:norm_equivalence},
    \begin{align*}
        \sup_{0\ne f\in \widetilde{W}^{\alpha,p}_{\rho_{\beta}}(\R)} &
         \frac{|I_{\rho_\upsilon}(f)-A_{\rho_\upsilon,n}(f)|}{\|f\|_{\widetilde{W}^{\alpha,p}_{\rho_{\beta}}(\R)}} \\[1mm]
        & \ge C_{\beta, \alpha}^{-1} \,
         \sup_{0\ne f\in\mathring{W}^{\alpha,p}(0,1)}
        \frac{\bigg|{\displaystyle \int_0^1} f(x)\rho_{\upsilon}(x)\rd x -A_{\rho_{\upsilon},n}(f) \bigg|}{\|f\|_{\mathring{W}^{\alpha,p}(0,1)}}
        \\[1mm]
        &=
        C_{\beta, \alpha}^{-1} \,
        \sup_{0\ne g\in\mathring{W}^{\alpha,p}(0,1)}
        \frac{\bigg| {\displaystyle \int_0^1 } g(x)\rd x - A_{\rho_{\upsilon},n}(R_\upsilon^{-1} g) \bigg|}{\|R_\upsilon^{-1} g\|_{\mathring{W}^{\alpha,p}(0,1)}} \\[1mm]
       &\ge  \frac{\| R_\upsilon \|_{\mathscr{L}(\mathring{W}^{\alpha,p}(0,1))} }{ C_{\beta, \alpha}} \sup_{0\ne g\in\mathring{W}^{\alpha,p}(0,1)}  \frac{\bigg| {\displaystyle \int_0^1 } g(x)\rd x -A_{n}(g) \bigg|}{\|g\|_{\mathring{W}^{\alpha,p}(0,1)}},
        \end{align*}
where $A_n$ is a linear quadrature rule on $(0,1)$ with at most $n$ nodes (i.e., with those nodes of $A_{\rho_{\upsilon},n}$ that lie in the interval $(0,1)$) and corresponding weights $w_j/\rho_{\upsilon}(x_j)$. In consequence, the claim follows from \eqref{eq:lowerbound0}.
\end{proof}

Although one could replace $\widetilde{W}^{\alpha,p}_{\rho_{\beta}}(\R)$ by the smaller space $W^{\alpha,p}_{\rho_{\beta}}(\R)$ in Proposition~\ref{prop:lower_bound} without essentially altering its proof, the proposition cannot be trivially extended to prove the worst-case optimality of Theorem~\ref{thm:main_result2} for fractional smoothness indices $\alpha > 1/p$: 
\cite[Remark~2]{T1990} employed in the proof assumes an integer smoothness index.

It should be noted that Kuo, Plaskota and Wasilkowski have proved a lower bound \cite[Theorem~3]{KPW2016} for a related weighted function approximation problem, applicable also to numerical integration as pointed out in \cite[Remark~6]{KPPW2020}. The lower bound exhibits similar behavior in the number of quadrature nodes as ours \eqref{eq:lower-bound}, i.e., $n^{-\alpha}$ with $\alpha$ being a smoothness index for the considered function spaces that differ from those used in this work. Another notable difference in our setting is that we do not assume integrability or non-increasing property of the weights themselves. Although we assume the proofs in~\cite{KPW2016,KPPW2020} could be modified to our setting, we considered presenting an independent proof for Proposition~\ref{prop:lower_bound} a more straightforward approach.

\section{Polynomial exactness for $\omega$}
\label{sec:exact_omega}
Unlike Gaussian quadrature rules that maximize the degree of polynomial exactness, our method is designed for achieving the optimal rate of convergence. Nevertheless, our method is also polynomially exact for the basic weight \eqref{eq:weight0}, $\gamma = 1$ in \eqref{eq:variable_change} and certain polynomial degrees. This property is inherited from the exactness of the trapezoidal rule for trigonometric polynomials on the unit circle. We consider this an independently interesting property that merits separate discussion in this subsection.

\begin{proposition}
    Let $n\in\N$, $\gamma = 1$ and $\upsilon$ be an even positive integer that satisfies $\upsilon \leq 2n$. Then, for every nonnegative integer $m \le \upsilon - 2$,
    \[
    \int_{\R} x^m \omega_{\upsilon}(x) \rd x=
    Q_{\omega_\upsilon,n}(x^m),
    \]
    where the quadrature rule $Q_{\omega_\upsilon,n}$ is defined by \eqref{eq:quadrature} with $\rho_\upsilon = \omega_\upsilon$ and $\gamma = 1$ in the definition of $\phi = \phi_\gamma$.
\end{proposition}
\begin{proof}
By our change of variables $x= \phi_1(\theta) = -\cot(\theta/2)$, we have (cf.~\eqref{eq:invphi})
    \begin{align*}
        \int_{\R} x^m \omega_{\upsilon}(x) \rd x
       &=
      \frac{1}{2} \int_{0}^{2\pi} \bigg(-\frac{\cos(\theta/2)}{\sin(\theta/2)}\bigg)^m (\sin(\theta/2))^{\upsilon-2} \rd \theta
       \\&=
       \frac{(-1)^m}{2} \int_{0}^{2\pi} (\cos(\theta/2))^m (\sin(\theta/2))^{\upsilon-2-m} \rd \theta.
    \end{align*}
By virtue of standard trigonometric identities,
\begin{align*}
     (\cos&(\theta/2))^m  (\sin(\theta/2))^{\upsilon-2-m} 
 \\[2mm]
 &= \frac{1}{2^{\upsilon/2 -1}}
\left\{
\begin{array}{ll}
     \big(1 + \cos(\theta) \big)^{m/2} \big(1 - \cos(\theta) \big)^{(\upsilon - 2 - m)/2} & \text{if $m$ is even},\\[1mm]
     \sin ( \theta ) \big(1 + \cos(\theta) \big)^{(m-1)/2} \big(1 - \cos(\theta) \big)^{(\upsilon - 3 - m)/2} & \text{if $m$ is odd}, 
     \end{array}
     \right.
\end{align*}
which is a trigonometric polynomial of degree $\upsilon/2 - 1$ (note that by assumption, $m \le \upsilon - 3$ if $m$ is odd). The $n$-point trapezoidal rule is exact for integrating trigonometric polynomials of degree $n-1$ over the unit circle (see,~e.g.,~
\cite[Equation~(3.8.12)]{Atkinson_1989_book}), which is essentially a consequence of the fact that the regular polygon formed by equidistant points on the unit circle always has its center at the origin. Hence, the claim follows from the definition of the M\"obius-transformed trapezoidal rule in \eqref{eq:quadrature}. 
\end{proof}

\section{Numerical experiments}
\label{sec:numerics}
In this section, we demonstrate the M\"obius-transformed trapezoidal rule, with the choice $\gamma =1$ in \eqref{eq:variable_change}, by two numerical experiments employing the standard weight function $\omega_\upsilon$ from \eqref{eq:weight0}. The first experiment considers a finitely smooth integrand function 
$$
f_1(x)=|x|\cos(x+1),
$$ 
with the goal to numerically observe the interplay between the Sobolev smoothness and behavior at infinity. The second experiment studies the smooth function 
$$
f_2(x)=(x^4+x^2+x+1)^{1/4},
$$ 
whose derivatives are more well-behaved at infinity than $f_2$ itself (cf.~Lemma~\ref{lemma:basic}). This allows to numerically confirm that resorting to the spaces $\widetilde{W}_{\omega_\kappa}^{\alpha,p}(\R)$, instead of $W_{\omega_\kappa}^{\alpha,p}(\R)$, enables predicting more accurately the convergence rate of the M\"obius-transformed trapezoidal rule. We will also use $f_2$ as an example for the cancellation phenomenon mentioned in Remark~\ref{remark:cancel}.

Throughout this section, we use Matlab 2025b with double precision arithmetic, except when we calculate the reference value of the true integral using Mathematica.

\subsection{Finitely smooth integrand}
Since the absolute value function belongs to $W^{s,p}_{\rm loc}(\R)$ with $1 \leq p < \infty$ if and only if $s<1+1/p$, as can be deduced directly from the Sobolev--Slobodeckij formulation for the fractional Sobolev norm (see, e.g., \cite{Leoni2023} and \cite[Theorem~7.57(c)]{Adams1975}), and cosine is infinitely smooth with bounded derivatives, it follows that $f_1 \in W^{s,p}_{\omega_\kappa}(\R)$ for $1 \leq p < \infty$, $s<1+1/p$, and $\kappa > 1 + p$. As the derivatives of $f_1$ exhibit similar growth at infinity as $f_1$ itself, there is no extra benefit in employing the variable-weight spaces $\widetilde{W}_{\omega_{\kappa'}}^{s,p}(\R)$, and thus we resort to Theorem~\ref{thm:main_result2} when interpreting the observed numerical convergence rates.

Consider evaluating the integral
\begin{equation}
\label{eq:test1}
I_{\omega_\upsilon}(f_1) = \int_{\R} f_1(x) \omega_{\upsilon}(x)\rd x
\end{equation}
numerically. According to Theorem~\ref{thm:main_result2}, the M\"obius-transformed trapezoidal rule exhibits convergence rate $n^{-\alpha}$ for this task if $f_1$ belongs to $W^{\alpha,p}_{\omega_{p(\upsilon - 2  - 2\alpha)+2}}(\R)$ for any $p \in (1, \infty)$. Since $f_1 \in W^{s,p}_{\omega_\kappa}(\R)$ for $s < 1 + 1/p$ and $\kappa > 1+p$, this gives two conditions for $\alpha$:
\[
\alpha < 1 + 1/p \qquad \text{and} \qquad 1+p < p(\upsilon - 2  - 2\alpha)+2,
\]
or equivalently,
\[
\alpha < 1 + 1/p \qquad \text{and} \qquad \alpha < \frac{\upsilon - 3}{2} + \frac{1}{2p}.
\]
Letting $p$ tend to $1$ reveals that the convergence rate predicted by Theorem~\ref{thm:main_result2} is $n^{-\alpha}$ for
\begin{equation}
\label{eq:alpha_test1}
\alpha < \min \! \bigg\{2, \frac{\upsilon - 2}{2} \bigg\}. 
\end{equation}

Figure~\ref{fig:abs-x-cosx} illustrates the convergence of the absolute error when numerically evaluating the integral $I_{\omega_\upsilon}(f_1)$ in \eqref{eq:test1} by the M\"obius-transformed trapezoidal rule for varying $\upsilon$. As we are not aware of the exact value for $I_{\omega_\upsilon}(f_1)$, the outputs of the M\"obius-transformed trapezoidal rule are compared to highly accurate numerical evaluations of $I_{\omega_\upsilon}(f_1)$ by Mathematica. The convergence plots in Figure~\ref{fig:abs-x-cosx} are well aligned with the prediction \eqref{eq:alpha_test1} by Theorem~\ref{thm:main_result2}.

\begin{figure}
\centering
 \includegraphics{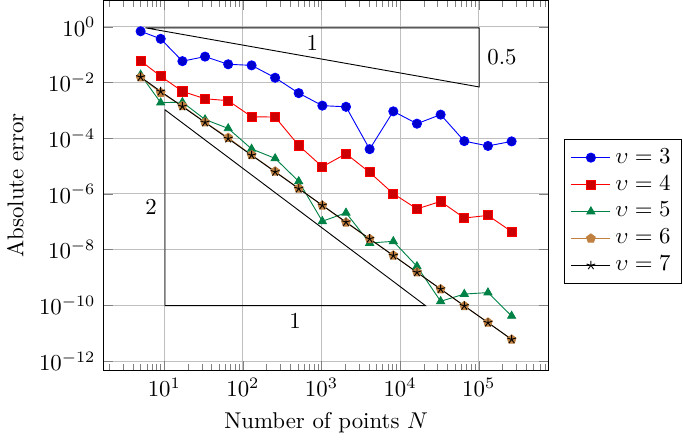}
\caption{Absolute integration error when applying the M\"obius-transformed trapezoidal rule to the integral $I_{\omega_\upsilon}(f_1)$ in \eqref{eq:test1} with the non-smooth integrand function $f_1(x)=|x|\cos(x+1)$. As the exact value of $I_{\omega_\upsilon}(f_1)$ is not known, the outputs of the M\"obius-transformed trapezoidal rule are compared to highly accurate evaluations of the integral by Mathematica.
}
\label{fig:abs-x-cosx}
\end{figure}

\subsection{Polynomial behavior at infinity}
Let us then consider numerically evaluating the integral 
\begin{equation}
\label{eq:test2}
I_{\omega_\upsilon}(f_2) = \int_{\R} f_2(x) \omega_{\upsilon}(x)\rd x.
\end{equation}
According to Lemma~\ref{lemma:basic}, 
$$
\big| f_2^{(\tau)} (x) \big| \leq C \omega_{-1 + \tau}(x), \qquad x \in \R, \ \tau \in \N_0.
$$
As $f_2$ is smooth, it follows straightforwardly from the definition \eqref{eq:sobolev-space2} that for any $\alpha \in \N_0$, $f_2 \in \widetilde{W}^{\alpha, p}_{\omega_\kappa}(\R)$ if and only if $\kappa > 1 + p (1-\alpha)$ --- observe that $f_2$ belongs to $W^{\alpha, p}_{\omega_\kappa}(\R)$ if and only if $\kappa > 1 + p$,~i.e.,~under a considerably stricter condition, which demonstrates the advantage in using the variable-weight Sobolev spaces. According to Theorem~\ref{thm:main_result}, the M\"obius-transformed trapezoidal rule should thus achieve the convergence rate $n^{-\alpha}$ for any $\alpha \in \N$ that satisfies
$$
1 + p (1-\alpha) < p (\upsilon - 2 - 2 \alpha) + 2 \quad \Longleftrightarrow \quad \alpha < \upsilon - 3 + \frac{1}{p}.
$$
Letting $p$ tend to $1$, we thus get the convergence rate $n^{-\alpha}$ for any positive integer satisfying $\alpha < \upsilon - 2$; see Remark~\ref{remark:case_p1}.

The convergence of the absolute error in the evaluation of the integral $I_{\omega_\upsilon}(f_2)$ by the M\"obius-transformed trapezoidal rule is presented in Figure~\ref{fig:poly-f2} for non-odd $\upsilon$ and in Figure~\ref{fig:poly-f2_2} for odd $\upsilon$. The reference values for the integrals are once again computed numerically with high accuracy by Mathematica. The convergence plots in Figure~\ref{fig:poly-f2} are in line with the rate hinted by Theorem~\ref{thm:main_result}, i.e.,~$n^{-\alpha}$ with $\alpha \approx \upsilon - 2$. (To be precise, Theorem~\ref{thm:main_result} does not actually predict this rate due to the exact condition being $\alpha < \upsilon - 2$ for an {\em integer} $\alpha$, as we have not considered interpolation of the $\widetilde{W}^{\alpha, p}_{\omega_\kappa}(\R)$ spaces to allow fractional smoothness indices.)

For odd $\upsilon > 2$, Figure~\ref{fig:poly-f2_2} seems to indicate exponential convergence, which is considerably faster than Theorem~\ref{thm:main_result} guarantees. This phenomenon can be explained by the following argumentation: Consider $m \in \R_+$ and let $P_{2d}$ be a positive polynomial of degree $2d$ for $d \in \N_0$, say, 
\[
P_{2d}(x) = \sum_{j=0}^{2d} \beta_j x^j, \qquad x \in \R.
\]
Then, by the standard change of variables $x = \phi_1(\theta) = -\cot(\theta/2)$,
\begin{align}
\label{eq:exponential}
    \int_\R & (P_{2d} (x))^{1/m} \omega_\upsilon(x) \rd x 
    =
    \frac{1}{2}\int_0 ^{2\pi} \big (P_{2d}(-\cot(\theta/2))\big)^{1/m} (\sin(\theta/2))^{\upsilon-2} \rd \theta
    \\[1mm] 
    &\kern+1cm =
     \frac{1}{2} \int_0 ^{2\pi} \big((\sin(\theta/2))^{2d} P_{2d}(-\cot(\theta/2))\big)^{1/m} (\sin(\theta/2))^{\upsilon-2-2d/m} \rd \theta,
     \nonumber
\end{align}
where
\begin{align}
\label{eq:exp_int}
(\sin(\theta/2))^{2d}  P_{2d}& (-\cot(\theta/2))   = \sum_{j=0}^{2d} \beta_j (\sin(\theta/2))^{2d - j} (-\cos(\theta/2))^j \\[1mm]
&= \frac{1}{2^d} \, \sum_{k=0}^d \beta_{2k} \big(1-\cos(\theta)\big)^{d-k} \big(1 + \cos(\theta)\big)^{k} \nonumber \\[0mm]
& \ \ \ \ -\frac{\sin(\theta)}{2^d} \, \sum_{k=0}^{d-1} \beta_{2k+1} \big(1-\cos(\theta)\big)^{d-k-1} \big(1+\cos(\theta) \big)^{k}
\end{align} 
is an infinitely smooth positive function on the one-dimensional torus $\T$. This means that the smoothness of the integrand function on the right-hand side of \eqref{eq:exponential} is solely determined by the factor $(\sin(\theta/2))^{\upsilon-2-2d/m}$, which is also infinitely smooth on $\T$ if $\upsilon-2-2d/m$ is a nonnegative even integer, as can be verified using the identity $2 \sin^2(\theta/2) = 1-\cos(\theta)$. 

In fact, interpreting the integrand on the right-hand side of \eqref{eq:exponential} as a function on the unit circle in the complex plane, it can be continued as a complex analytic function to an open annular neighborhood of the unit circle if $\upsilon-2-2d/m$ is a nonnegative even integer, which guarantees {\em exponential} convergence of the M\"obius-transformed trapezoidal rule when it is applied to \eqref{eq:exponential}. Indeed,
 the Fourier coefficients of an analytic function $h$ over the torus satisfy $|\widehat{h}(k)| \le C e^{-a|k|} $ for some positive constants $a$ and $C$; see,~e.g.,~\cite[Section~12, Lemma~1]{A1983}. Using the exactness of the $n$-point trapezoidal rule on the unit circle for trigonometric polynomials whose frequency is not a nonzero multiple of $n$, we have
\begin{align*}
    \bigg|\int_0^{2\pi} h(\theta)\rd \theta - \frac{1}{2\pi n}\sum_{j=1}^n h(\theta_j) \bigg|
    &=
     \frac{1}{2\pi n} \bigg| \sum_{j=1}^n \sum_{0\ne k\in\Z}\widehat{h}(nk)e^{{\rm i} nk \theta_j} \bigg| \\[1mm]
    & \le \frac{1}{2\pi} \sum_{0\ne k\in\Z} |\widehat{h}(nk)| 
   \le \frac{1}{2\pi}  \sum_{0\ne k\in\Z}C e^{-a|nk|} \\[1mm]
   &=\frac{C}{\pi}\frac{e^{-an}}{1-e^{-an}},
\end{align*}
demonstrating exponential convergence.
 
The integrand function $f_2$ can be written in the form $(P_{2d})^{1/m}$ for $d = 2$ and $m=4$, meaning that the above line of reasoning leads to exponential convergence of the M\"obius-transformed trapezoidal rule for the integral \eqref{eq:test2} if $\upsilon - 3$ is nonnegative and even,~i.e.,~if $\upsilon \geq 3$ is odd. This explains the exceptionally high numerical convergence rates in Figure~\ref{fig:poly-f2_2}.

\begin{figure}
\centering
 \includegraphics{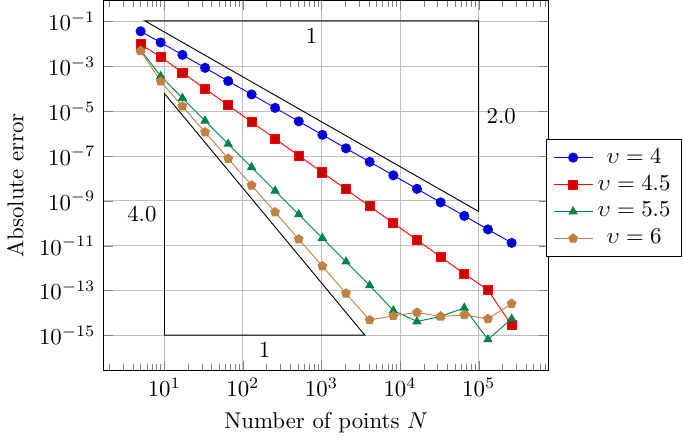}
\caption{Absolute integration error when applying the M\"obius-transformed trapezoidal rule to the integral $I_{\omega_\upsilon}(f_2)$ in \eqref{eq:test2} with the smooth integrand function $f_2(x)=(x^4+x^2+x+1)^{1/4}$ and non-odd $\upsilon$. As the exact value of $I_{\omega_\upsilon}(f_2)$ is not known, the outputs of the M\"obius-transformed trapezoidal rule are compared to highly accurate evaluations of the integral by Mathematica.
}
\label{fig:poly-f2}
\end{figure}
\begin{figure}
\centering
 \includegraphics{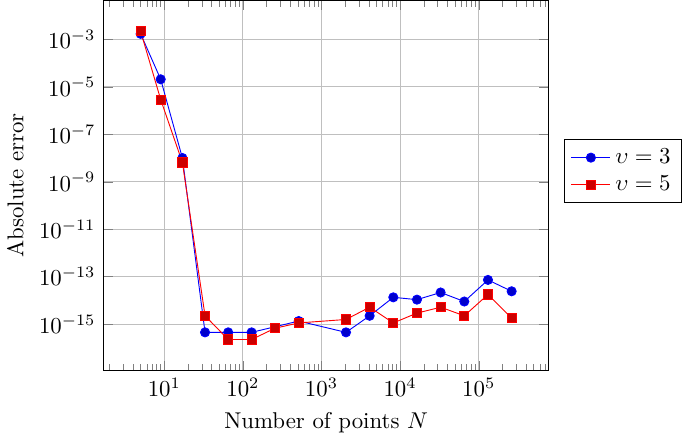}
\caption{Absolute integration error when applying the M\"obius-transformed trapezoidal rule to the integral $I_{\omega_\upsilon}(f_2)$ in \eqref{eq:test2} with the smooth integrand function $f_2(x)=(x^4+x^2+x+1)^{1/4}$ and odd $\upsilon$. As the exact value of $I_{\omega_\upsilon}(f_2)$ is not known, the outputs of the M\"obius-transformed trapezoidal rule are compared to highly accurate evaluations of the integral by Mathematica.
}
\label{fig:poly-f2_2}
\end{figure}

\section{Concluding remarks}
\label{sec:conclusion}
In this paper, we studied integration on the real line against polynomial weights. We proved that the Möbius-transformed trapezoidal rule achieves the optimal convergence rate for Sobolev spaces of integer smoothness: Theorem~\ref{thm:main_result} provides the upper bound, and Proposition~\ref{prop:lower_bound} gives the matching lower bound. Furthermore, by interpolating the underlying weighted Sobolev spaces, we extended the upper bound in a slightly weaker form to spaces of fractional smoothness; see Theorem~\ref{thm:main_result2}. The established convergence rates were verified through numerical experiments.

A natural extension of our theory is to consider multivariate integration. In this context, we are aware of recent interesting developments in high-dimensional numerical integration over $\R^d$, such as \cite{CDP2025,POH2025}, where quasi-Monte Carlo methods are used in combination with change of variables. Those results demonstrate dimension-independent convergence, but the cases considered are limited to prescribed smoothness of $\alpha=1$ or $\alpha=2$. In contrast, Kazashi, Goda, and one of the present authors showed in \cite{KSG2025} that quasi-Monte Carlo methods with M\"obius transformation achieve the optimal convergence rate $\calO(n^{-\alpha})$ up to a logarithmic factor, for Gaussian-weighted Sobolev spaces of arbitrary smoothness~$\alpha$, but with an underlying constant that increases exponentially with the dimension~$d$. We believe that M\"obius-transformed quasi-Monte Carlo methods can achieve similar convergence rates for polynomially-weighted Sobolev spaces over $\R^d$, but with the constant still cursed by the dimensionality. Aiming for dimension-independent optimal convergence of numerical integration in Sobolev spaces of arbitrary smoothness $\alpha$ provides an interesting avenue for future research.

\ifpreprint
\section*{Acknowledgments}
This work is supported by the Research Council of Finland (decisions 348503 and 359181).
\fi

\emergencystretch=2em
\printbibliography

\end{document}